\documentclass[oneside,english]{amsart}
\usepackage[T1]{fontenc}
\usepackage[latin9]{inputenc}
\usepackage{babel}
\usepackage{amstext}
\usepackage{amsthm}
\usepackage{amssymb}
\usepackage[unicode=true,
 bookmarks=false,
 breaklinks=false,pdfborder={0 0 1},backref=section,colorlinks=false]
 {hyperref}

\makeatletter
\numberwithin{equation}{section}
\numberwithin{figure}{section}
\theoremstyle{plain}
\newtheorem*{thm*}{\protect\theoremname}
\theoremstyle{remark}
\newtheorem*{rem*}{\protect\remarkname}
\theoremstyle{definition}
\newtheorem*{defn*}{\protect\definitionname}
\theoremstyle{plain}
\newtheorem*{conjecture*}{\protect\conjecturename}
\theoremstyle{plain}
\newtheorem{thm}{\protect\theoremname}[section]
\theoremstyle{plain}
\newtheorem{prop}[thm]{\protect\propositionname}
\theoremstyle{plain}
\newtheorem{lem}[thm]{\protect\lemmaname}
\theoremstyle{remark}
\newtheorem{rem}[thm]{\protect\remarkname}
\theoremstyle{plain}
\newtheorem{cor}[thm]{\protect\corollaryname}
\theoremstyle{definition}
\newtheorem{example}[thm]{\protect\examplename}

\DeclareMathOperator{\vol}{vol}

\DeclareMathOperator{\rad}{rad}

\DeclareMathOperator{\diam}{diam}

\DeclareMathOperator{\gexp}{gexp}

\DeclareMathOperator{\scal}{scal}

\DeclareMathOperator{\Ric}{Ric}

\DeclareMathOperator{\curv}{curv}

\makeatother

\providecommand{\conjecturename}{Conjecture}
\providecommand{\corollaryname}{Corollary}
\providecommand{\definitionname}{Definition}
\providecommand{\examplename}{Example}
\providecommand{\lemmaname}{Lemma}
\providecommand{\propositionname}{Proposition}
\providecommand{\remarkname}{Remark}
\providecommand{\theoremname}{Theorem}

\begin{document}
\title{Alexandrov Spaces with Maximal Radius}
\author{Karsten Grove}
\address{Department of Mathematics\\
 University of Notre Dame\\
 Notre Dame, IN 46556}
\email{kgrove2@nd.edu}
\thanks{The first author was supported in part by the National Science Foundation}
\author{Peter Petersen}
\address{520 Portola Plaza\\
 Dept of Math UCLA\\
 Los Angeles, CA 90095}
\email{petersen@math.ucla.edu}
\begin{abstract}
In this paper we prove several rigidity theorems related to and including
Lytchak's problem. The focus is on Alexandrov spaces with $\curv\geq1$,
nonempty boundary, and maximal radius $\frac{\pi}{2}$. We exhibit
many such spaces that indicate that this class is remarkably flexible.
Nevertheless, we also show that when the boundary is either geometrically
or topologically spherical, then it is possible to obtain strong rigidity
results. In contrast to this one can show that with general lower
curvature bounds and strictly convex boundary only cones can have
maximal radius. We also mention some connections between our problems
and the positive mass conjectures. This paper is an expanded version
and replacement of \cite{GPold}.
\end{abstract}

\maketitle
\global\long\def\bbR{\mathbb{\mathbb{R}}}%
\global\long\def\span{\mathrm{span}}%
\global\long\def\rank{\mathrm{rank}}%
\global\long\def\bbC{\mathbb{\mathbb{C}}}%
\global\long\def\bbS{\mathbb{\mathbb{S}}}%
\global\long\def\bbZ{\mathbb{\mathbb{Z}}}%
\global\long\def\bbF{\mathbb{\mathbb{F}}}%
\global\long\def\bbP{\mathbb{\mathbb{P}}}%
\global\long\def\abs#1{\left|#1\right|}%
\global\long\def\norm#1{\left\Vert #1\right\Vert }%
\global\long\def\pitwo{\frac{\pi}{2}}%

\section*{Introduction}

It is a basic fact that any Alexandrov space $X$ with lower curvature
bound $1$, $\curv X\geq1$, has diameter, $\diam X\leq\pi$ and hence
radius, $\rad X\leq\pi$. Moreover, in case of equality $X$ is a
spherical suspension in the first case and the unit sphere in the
second. With a non-positive lower curvature bound $k\le0$ no such
upper bounds for diameter or radius exist in general. When $X$ has
non-empty boundary $\partial X$ and $\curv X\geq1$, then its radius
is further restricted to $\rad X\leq\pitwo$. In fact, if $\rad X>\pitwo$,
then $X$ is homeomorphic to a sphere \cite{GP}, \cite[Corollary 5.2.2]{Pe}.
Again, in the case of lower curvature bound $k\le0$, there are no
such bounds. It is, however, possible to control the radius when the
boundary is $\lambda$-convex. In this case we will see that indeed
there is an $r=$ $r(k,\lambda)$ such that $\rad X\leq r\left(k,\lambda\right)$,
where $r(1,\lambda)\nearrow\pitwo$ as $\lambda\searrow0$.

Our aim is to establish rigidity theorems for Alexandrov spaces with
boundary and maximal radius in the situations described above.

The most diverse case is when $\curv X\geq1$ and $\rad X=\pitwo$.
The simplest examples are spherical joins $E*S$, where $E$ and $S$
are Alexandrov spaces with $\curv\geq1$, $\rad E\geq\pitwo$, and
$\partial S\neq\emptyset$. A special case is when $S$ is a point
and $E*S$ becomes the spherical cone over $E$. However, there are
many more examples. They will be classified in dimensions at most
four in corollary \ref{cor:leq4} using the Topological Regularity
Theorem below.

In general dimensions the following result is very useful for establishing
several interesting rigidity theorems. Recall that when $X$ has $\curv\geq1$
and and boundary $\partial X\neq\emptyset$, then there is a unique
point $s$ at maximal distance $\leq\pitwo$ from $\partial X$, called
the \emph{soul} point of $X$. Also, a point is called regular when
its tangent cone is Euclidean.
\begin{thm*}[Inner Regularity]
Let $X$ be an $n$-dimensional Alexandrov space with $\curv\geq1$
and $\partial X\neq\emptyset$. If $\rad X=\frac{\pi}{2}$ and the
soul of $X$ is a regular point, then $X$ is isometric to a spherical
join $\mathbb{S}^{k}\left(1\right)*S$ , where $S$ is an $\left(n-k-1\right)$-dimensional
Alexandrov space with $\curv\geq1$, nonempty boundary, and $\rad S<\frac{\pi}{2}$. 
\end{thm*}
A natural problem raised by Lytchak asks whether an $n$-dimensional
Alexandrov space with curvature $\geq1$ has the property that its
boundary has volume $\leq\vol\mathbb{S}^{n-1}\left(1\right)$. Petrunin
answered this in the affirmative in \cite[section 3.3.5]{Pe}. Lytchak
further asked what happens when the boundary has maximal volume \cite{Ly}.
Obviously the hemisphere is an example, but so is the intersection
of two hemispheres making an angle $\alpha<\pi$. We refer to this
as an \emph{Alexandrov lens} and denote it by $L_{\alpha}^{n}$ (note
that $L_{\pi}^{n}$ is the hemisphere). The above result can be used
to prove that these exhaust all possibilities.
\begin{thm*}[Maximal Volume]
Let $X$ be an $n$-dimensional Alexandrov space with $\curv\geq1$
and $\partial X\neq\emptyset$. If $\vol\partial X=\vol\mathbb{S}^{n-1}\left(1\right)$,
then $X$ is isometric to $L_{\alpha}^{n}$ for some $0<\alpha\leq\pi$. 
\end{thm*}
This theorem shows that $\partial X$ with its induced inner metric
is isometric to the unit sphere $\mathbb{S}^{n-1}$. In the special
case where $X$ is a leaf space, this result was proved in \cite{GMP}
and played a key role in confirming the boundary conjecture for Alexandrov
spaces that happen to be leaf spaces as well. 

The Inner Regularity Theorem can with the help of \cite{RW} be extended
in several ways. First we have a version with slightly more flexibility
than the above join examples and with purely topological assumptions
about the boundary.
\begin{thm*}[Topological Regularity]
Let $X^{n}$ be an Alexandrov space with $\curv\ge1$, $\partial X\neq\emptyset$,
and $\rad X=\pitwo$. If $\partial X$ is a topological manifold and
a $\bbZ_{2}$-homology sphere, then the double is a finite quotient
of a join: $D\left(X\right)=\left(A*B\right)/G$. Here $G$ is a finite
group acting effectively and isometrically on both $A$ and $B$ whose
action is extended to the spherical join $A*B$.
\end{thm*}
Since $\partial X$ and the space of directions at the soul, $S_{s}X$,
are homeomorphic we could, like in the Inner Regularity case, have
made the topological restrictions on $S_{s}X$ instead. This helps
us show that all examples in dimensions $\leq4$, in fact, satisfy
the conclusion of this theorem even though not all satisfy the topological
assumption. This theorem also suggests that weaker geometric assumptions
than those from the Inner Regularity Theorem might be used to obtain
rigidity. The next result offers a very general geometric condition
that guarantees that the space is a spherical join.
\begin{thm*}[Weak Inner Regularity]
Let $X^{n}$ be an Alexandrov space with $\curv\ge1$, $\partial X\neq\emptyset$,
and $\rad X=\pitwo$. If $\rad S_{s}X>\pitwo$, then $X=\hat{E}*\hat{S}$,
where $\rad\hat{E}>\pitwo$ and $\rad S_{s}\hat{S}>\pitwo$. 
\end{thm*}
\begin{rem*}
Note that the concluding geometric restrictions imply that $\hat{E}$
is topologically a sphere and that $\hat{S}$ is topologically a closed
disk.
\end{rem*}
\bigskip{}

We now turn to the case where $\curv X\geq k$, in particular allowing
$k\leq0$. As already mentioned this does not yield an upper bound
on the radius of $X$. However, when the boundary is strictly convex
in the sense studied \cite{A-B}, such a bound does exist. Specifically,
one can \emph{quantify convexity} of the boundary of an Alexandrov
by comparing with strictly convex model spaces. The closed $r_{0}$-ball
$\bar{B}_{k}\left(r_{0}\right)$ in the simply connected space form
of constant curvature $k$ has a boundary that is totally umbilic:
$\mathrm{II}_{\partial B}=\lambda_{0}g_{\partial B}$, where $\lambda_{0}=\lambda_{0}\left(r_{0},k\right)$.
Specifically: 
\begin{eqnarray*}
\lambda_{0}\left(r_{0},0\right) & = & \frac{1}{r_{0}},\\
\lambda_{0}\left(r_{0},1\right) & = & \cot r_{0},\\
\lambda_{0}\left(r_{0},-1\right) & = & \coth r_{0}.
\end{eqnarray*}

\begin{defn*}
We say that an Alexandrov space $X$ has \emph{$\lambda_{0}$-convex
boundary}, where $\lambda_{0}>0$, provided: for each $x\in\mathrm{int}X$
and $p\in\partial X$ with $\abs{xp}=\abs{x\partial X}$ we have

\[
\abs{pq}\cos\left(\angle\left(\overrightarrow{px},\overrightarrow{pq}\right)\right)-\frac{\lambda_{0}}{2}\abs{pq}^{2}\geq o\left(\abs{pq}^{2}\right)
\]
for all $q\in\partial X$ sufficiently near $p$. \smallskip{}

Here $\overrightarrow{pq}\in S_{p}X$ denotes the direction at $p$
of a minimal geodesic from $p$ to $q$ and of length $|pq|$. The
law of cosines shows that $\bar{B}_{k}\left(r_{0}\right)$ has $\lambda_{0}$-convex
boundary in this sense. 
\end{defn*}
This type of quantified convexity was studied in detail in \cite{A-B}.
We shall use it to solve the analogue of Lytchak's problem for a general
lower curvature bound and strictly convex boundary. This was also
done with slightly different techniques in \cite{GL} assuming that
$k\geq0$. 
\begin{thm*}
Let $X$ be an Alexandrov space with curvature $\geq k$ and $\lambda_{0}$-convex
boundary, where $\lambda_{0}^{2}>\max\left\{ -k,0\right\} $. If $r_{0}$
is defined by $\lambda_{0}\left(r_{0},k\right)=\lambda_{0}$, then
$\rad X\leq r_{0}$ and $\vol\partial X\leq\vol\partial\bar{B}_{k}\left(r_{0}\right)$.
Moreover, if $\rad X=r_{0}$, then $X$ is isometric to a cone $C_{k}Y$
with constant radial curvature $k$; and if $\vol\partial X=\vol\partial\bar{B}_{k}\left(r_{0}\right)$,
then $X$ is isometric to $\bar{B}_{k}\left(r_{0}\right)$. 
\end{thm*}
The questions discussed so far are related to another circle of ideas
that come from the positive mass conjectures. The original general
version first formulated and proved by Miao in \cite{Miao} is similar
to the theorem just mentioned above.
\begin{conjecture*}
If $(M,g)$ is a Riemannian $n$-manifold with $\scal\geq0$, $\partial M=S^{n-1}(1)$,
and $\Pi_{\partial M}\geq g_{\partial B(0,1)}$, then $M=B(0,1)\subset\mathbb{R}^{n}$. 
\end{conjecture*}
Min-Oo in \cite{Min-Oo} established the hyperbolic equivalent and
also proposed a version for positive curvature. 
\begin{conjecture*}[Min-Oo \cite{Min-Oo}]
 If $(M,g)$ is a Riemannian $n$-manifold with $\scal\geq n(n-1)$,
$\partial M=S^{n-1}(1)$, and $\Pi_{\partial M}\geq0$, then $M$
is a hemisphere. 
\end{conjecture*}
Brendle, Marques, and Neves in \cite{BMN}, however, found a counter
example to this conjecture. But just prior to this example Hang and
Wang in \cite{HW} proved the following version. 
\begin{thm*}
If the scalar curvature assumption is replaced with the stronger condition:
$\Ric\geq n-1$, then the conclusion of Min-Oo's conjecture holds. 
\end{thm*}
It is worth noting that this theorem is indeed extremely sensitive
to the condition that the boundary be smooth. Even with the much stronger
condition that $\sec\geq1$ the Alexandrov lens $L_{\alpha}^{n}$
is an example whose boundary is intrinsically isometric to $\bbS^{n-1}\left(1\right)$.
Note, however, that the boundary is only convex, not strictly convex
in the sense studied above. Nevertheless, it is not clear if something
like Lytchak's problem is true for Riemannian $n$-manifolds with
$\Ric\geq n-1$ and nonempty convex boundary.

The proofs employ several important techniques from the theory of
Alexandrov spaces that are explained in section 1. It is interesting
to note that some of these concepts appear to be necessary even when
the interiors of the spaces are Riemannian manifolds. Section 2 contains
some preliminary results for Alexandrov spaces with positive curvature
and maximal radius. In section 3 we offer several examples that indicate
how intricate and complex such spaces can be. Section 4 includes the
proof of the Inner Regularity Theorem. This is used to resolve Lytchak's
problem as well as another result related to the main focus of \cite{GL2}.
Section 5 is focused on establishing the Topological Regularity Theorem,
which in turn is used to establish the Weak Inner Regularity Theorem
and a complete classification of positively curved spaces with maximal
radius in low dimensions. Section 6 contains a short account of the
last theorem about Alexandrov spaces with strictly convex boundary.

It is worth pointing out that it is unknown whether the boundary of
an Alexandrov space is a priori an Alexandrov space, so some care
must be taken when working with conditions that pertain to the boundary.

The authors would like to thank Alexander Lytchak for several helpful
comments on the previous version of this paper, \cite{GPold}, and
Vitali Kapovitch for offering a very efficient proof of theorem 1.1.

\section{Alexandrov Geometry Preliminaries}

In this section we establish notation and explain the important constructions
from Alexandrov geometry that are needed. The book \cite{BBI} explains
all of the basic notions and the survey article \cite{Pe} covers
a number of more advanced topics including the gradient exponential
map.

The distance between points $p,\,q$ in a metric space $X$ is denoted
$\left|pq\right|$, or $\left|pq\right|_{X}$ if confusion is possible.
In our case this will always be an intrinsic or inner distance that
measures the length of a shortest path joining the points. For an
Alexandrov space $T_{p}X$ denotes the tangent cone at $p\in X$ and
$S_{p}X\subset T_{p}X$ the space of (unit) directions. Our geodesics
and quasi-geodesics are always parametrized by arclength. The notation
$\overset{\longrightarrow}{pq}\in S_{p}X$ refers to the direction
of a minimal geodesic from $p$ to $q$ and $\overset{\Longrightarrow}{pq}\subset S_{p}X$
is the space of all such directions.

If $Y$ is an Alexandrov space with $\curv Y\geq1$, then the \emph{curvature
$k$ cone over $Y$}, denoted $C_{k}Y$, is the cone $Y\times[0,\infty)/(Y\times\{0\})$
(with $\infty$ replaced by $\frac{\pi}{2\sqrt{k}}$ when $k>0$)
equipped with the metric where $|(y,t)(y,s)|=|t-s|$ and $|(y_{1},t)(y_{2},s)|$
is the distance in the curvature $k$ plane between the end points
of a hinge with angle $\abs{y_{1}y_{2}}_{Y}$ and side lengths $t$
and $s$. $C_{k}Y$ is an Alexandrov space with curvature bounded
below by $k$. Note that when $k>0$ it is possible to define the
cone on $Y\times[0,\frac{\pi}{\sqrt{k}})/(Y\times\{0\})$, but this
space is not complete or geodesically convex.

The \emph{spherical suspension} $\Sigma_{1}Y$ of $Y$ is simply the
double of the spherical cone. It is also the space of directions of
$C_{0}Y\times\mathbb{R}$ (equipped with the product metric) at $(c,0)$
where $c$ is the cone point of $C_{0}Y$.

In general, given two Alexandrov spaces $X$ and $Y$ with $\curv\geq1$
the metric product $C_{0}X\times C_{0}Y$ is an Alexandrov space with
non-negative curvature and its space of directions at the cone point
$(c_{1},c_{2})$ is the \emph{spherical join} $X*Y$. The distances
are given by
\[
\cos|(x_{1},r_{1},y_{1})(x_{2},r_{2},y_{2})|=\cos(r_{1})\cos(r_{2})\cos(|x_{1}x_{2}|)+\sin(r_{1})\sin(r_{2})\cos(|y_{1}y_{2}|).
\]
Note that $\Sigma_{1}Y=\left\{ 0,\pi\right\} *Y$.

\medskip{}

Let $X$ be a compact Alexandrov space with $\mathrm{curv}\geq k$.
The \emph{gradient exponential map} at $p\in X$ 
\[
\mathrm{gexp}_{p}\left(k;\cdot\right):T_{p}X\rightarrow X
\]
is defined on the tangent cone when $k\leq0$ and on the closed ball
$\bar{B}\left(o_{p},\frac{\pi}{2\sqrt{k}}\right)\subset T_{p}X$ when
$k>0$. We can identify this domain with the cone $C_{k}(S_{p}X)$.
With this new metric the gradient curves can be reparametrized so
that $\mathrm{gexp}_{p}\left(k;\cdot\right):C_{k}(S_{p}X)\rightarrow X$
becomes distance nonincreasing. 

Along a radial curve in $T_{p}X$ the gradient exponential map follows
the direction of maximal increase for the distance to $p$. Thus,
it follows minimal geodesics until they hit cut points. In general,
it moves in the direction of a point in $S_{q}X$ that is at maximal
distance from $\overset{\Longrightarrow}{qp}$ and at a rate that
is specified by both $\left|pq\right|$ and how far $\overset{\Longrightarrow}{qp}$
spreads out. Flow lines terminate at critical points $q$, i.e., when
$\overset{\Longrightarrow}{qp}$ forms a $\frac{\pi}{2}$-net in $S_{q}X$.
Finally, the gradient exponential map is distance nonincreasing and
\[
\mathrm{gexp}_{p}\left(k;\bar{B}\left(o_{p},r\right)\right)=\bar{B}\left(p,r\right)
\]
for all $r$ with the caveat that $r\leq\frac{\pi}{2\sqrt{k}}$ when
$k>0$.

We shall also be using \emph{quasi-geodesics}. They have several nice
properties. Unlike geodesics they can be defined for all time and
there is a quasi-geodesic in each direction of the space.

The left and right derivatives of a unit speed curve, if they exist,
are defined as 
\[
\dot{c}^{+}\left(t_{0}\right)=\lim_{t\rightarrow t_{0}^{+}}\overrightarrow{c\left(t_{0}\right)\,c\left(t\right)}
\]
\[
\dot{c}^{-}\left(t_{0}\right)=\lim_{t\rightarrow t_{0}^{-}}\overrightarrow{c\left(t_{0}\right)\,c\left(t\right)}
\]
Note that even for a differentiable curve in a manifold they point
in opposite directions. Quasi-geodesics always have right and left
derivatives. 

Functions similarly can have right and left derivatives when restricted
to curves. Moreover, when a distance function is evaluated on a unit
speed curve then the left and right derivatives are always defined.

Perel'man's stability theorem also gives us the following result for
Alexandrov spaces with boundary.
\begin{thm}
Consider a compact Alexandrov space $X$ with $\partial X\neq\emptyset$.
If the distance $r\left(x\right)=\abs{x\partial X}$ has a unique
maximum at a soul $s\in X$ and no other critical points, then $S_{s}X$
and $\partial X$ are homeomorphic.
\end{thm}

\begin{proof}[Proof by V. Kapovitch]
From \cite[lemmas 4.7 and 5.2]{K1} one obtains a strictly concave
function $g$ near $s$ which has unique maxium at $s$ and whose
rescaled level sets are homeomorphic to $S_{s}X$. On the other hand
by Perel'man's fibration theorem (see \cite{K2} and \cite[section 8 property 7]{Pe})
we have that the level sets of $r$ that are near the soul are homeomorphic
to $\partial X$. The interpolated functions $r_{\epsilon}=\left(1-\epsilon\right)r+\epsilon g$
are also strictly concave near $s$ with a unique maximum at $s$
for all $\epsilon\in\left[0,1\right]$. This is a continuous family
in $\epsilon$ so by Perel'man's fibration theorem applied to $\left(x,\epsilon\right)\mapsto\left(r_{\epsilon}\left(x\right),\epsilon\right)$
the level sets are homeomorphic for all $\epsilon$. Considering $\epsilon=0,1$
we see that the level sets of $r$ near $s$ are homeomorphic to $S_{s}X$. 
\end{proof}

\section{Basic structure}

Except for section 6 we will only consider Alexandrov spaces $X$
with curvature $\geq1$. In this section we list some basic properties
for such spaces when they have $\rad\geq\pitwo$.
\begin{prop}
\label{prop:rad_susp}If $\mathrm{rad}X\geq\frac{\pi}{2}$ and $\mathrm{curv}X\geq1$,
then either $\mathrm{rad}S_{x}X\geq\frac{\pi}{2}$ for all $x\in X$
or $X=\Sigma_{1}S_{p}X$, where $\mathrm{rad}S_{p}X<\frac{\pi}{2}$. 
\end{prop}

\begin{proof}
Assume $\rad S_{p}X<\frac{\pi}{2}$ and select a quasi-geodesic $c:\left[0,\pi\right]\rightarrow X$
such that $c\left(0\right)=p$ and $\angle\left(\dot{c}\left(0\right),w\right)<\frac{\pi}{2}$
for all $w\in S_{p}X$. Let $q=c\left(\frac{\pi}{2}\right)$. Then
$\abs{xq}<\frac{\pi}{2}$ unless $\abs{px}=0,\,\pi$. In the latter
case we are finished. So if no such $x$ exists then $p$ is the one
and only point at distance $\frac{\pi}{2}$ from $q$. This shows
that $X-B\left(p,t\right)\subset\bar{B}\left(c\left(t\right),\frac{\pi}{2}-\delta\left(t\right)\right)$
for $t<\frac{\pi}{2}$ and $t$ near $\frac{\pi}{2}$. As $c:\left[0,\frac{\pi}{2}\right]\rightarrow X$
is now a geodesic and $\angle\left(\dot{c}\left(0\right),w\right)<\frac{\pi}{2}$,
for all $w\in S_{p}X$, it follows from Toponogov comparison that
$\bar{B}\left(p,t\right)\subset\bar{B}\left(c\left(t\right),\frac{\pi}{2}-\epsilon\left(t\right)\right)$
for $t<\frac{\pi}{2}$. This shows that $X\subset B\left(c\left(t\right),\frac{\pi}{2}\right)$
for some $t<\frac{\pi}{2}$ and contradicts that $\mathrm{rad}X\geq\frac{\pi}{2}$. 
\end{proof}
\medskip{}

A set $C\subset Y$ in an Alexandrov space is\emph{ $\pi$-convex}
if any geodesic of length $<\pi$ and whose end points lie in $C$
is entirely contained in $C$. This notion of convexity should not
be confused with the definition of \emph{boundary convexity} discussed
in the introduction and section 6. 

\medskip{}

In case $X$ has boundary it will have a unique soul $s\in X$ at
maximal distance from the boundary and well known comparison arguments
imply that $X\subset\bar{B}\left(s,\frac{\pi}{2}\right)$ (see also
\cite{Pe}). This shows that $\mathrm{gexp}_{s}\left(1;\cdot\right):C_{1}\left(S_{s}X\right)\rightarrow X$
is onto. Since this map is also distance nonincreasing it follows
that $\rad S_{s}X\geq\frac{\pi}{2}$ when $\rad X=\pitwo$ (see also
proposition \ref{prop:rad_susp}).

In the extreme case where $\rad X=\frac{\pi}{2}$ we define the \emph{edge}
of $X$ to be the dual set to $s$:
\[
E=\left\{ x\in X\mid\abs{xs}=\frac{\pi}{2}\right\} .
\]
It has the following important property.
\begin{prop}
If $X$ has boundary and $\rad X=\pitwo$, then $s$ is at maximal
distance from $E$.
\end{prop}

\begin{proof}
We first show that when $s$ is a critical point for the distance
to $E$, then $s$ is in fact at maximal distance from $E$. To see
this select $x\in X$ and a geodesic direction $\overrightarrow{sx}\in S_{s}X$.
If $s$ is critical for $E$, then there is a geodesic direction $\overrightarrow{se}\in S_{s}X$
with $e\in E$ such that $\angle\left(\overrightarrow{sx},\overrightarrow{se}\right)\leq\pitwo$.
Toponogov comparison then implies that $\abs{xe}\leq\pitwo$ as $\abs{xs},\abs{es}\leq\pitwo$.

Thus we need to show that when $s$ is not critical, then $\rad X<\pitwo$.
For that choose a unit speed geodesic $c\left(t\right)$ such that
$\dot{c}\left(0\right)\in S_{s}X$ forms an angle $<\pitwo-\epsilon$
with every direction $\overrightarrow{se}\in S_{s}X$ and $e\in E$.
We can now find an open neighborhood $U\supset E$ such that the directions
$U^{\prime}\subset S_{s}X$ for minimal geodesics from $s$ to points
in $U$ also form an angle $<\pitwo-\epsilon$ with $\dot{c}\left(0\right)$.
By compactness there exists $\delta>0$ such that $\abs{sz}\leq\pitwo-\delta$
for all $z\notin U$. So we can in addition fix $t$ such that $\abs{c\left(t\right)z}\leq\pitwo-\frac{\delta}{2}$
for all $z\notin U$. On the other hand if $\overrightarrow{sy}\in S_{x}X$
denotes a direction to a $y\in U$, then Toponogov comparison implies
\[
\cos\abs{c\left(t\right)y}\geq\cos t\cos\abs{sy}+\sin t\sin\abs{sy}\cos\left(\pitwo-\epsilon\right).
\]
Here the left-hand side is uniformly positive for any fixed small
$t$, so there is an $\epsilon_{1}>0$, such that $\abs{c\left(t\right)y}\leq\pitwo-\epsilon_{1}$
for all $y\in U$. This shows that $\rad X<\pitwo$.
\end{proof}
We can now define the \emph{spine} as the dual set to $E$ by
\[
S=\left\{ x\in X\mid\abs{xE}=\frac{\pi}{2}\right\} .
\]
By standard comparison, the complement of any open $\pitwo$ ball
in an Alexandrov space $X$ with $\curv X\geq1$ is $\pi$-convex
(even relative to $\partial X$ if contained in there, since geodesics
in the boundary are quasi-geodesics in $X$). Moreover, another simple
comparison argument shows that the distance functions to $E$ and
to $S$ have no critical points in $X-(E\cup S)$. In summary
\begin{prop}
\label{prop:dual} Assume that $X$ has $\curv\geq1$, nonempty boundary
and $\rad X=\pitwo$. It follows that 
\begin{enumerate}
\item $E\subset X$ is closed and $\pi$-convex in both $X$ and $\partial X$. 
\item $S$ is closed, $\pi$-convex in $X$, and $\rad S<\frac{\pi}{2}$. 
\item $E$, respectively $S$, is a deformation retract of $X-S$, respectively
of $X-E$ 
\item $E$, respectively $S\cap\partial X$, is a deformation retract of
$\partial X-S\cap\partial X$, respectively of $\partial X-E$ 
\end{enumerate}
Concrete deformations are provided by the gradient flows for the distance
functions to $S$ and $E$ respectively, preserving the extremal set
$\partial X$. 
\end{prop}

We will now see that $E\ne\partial X$ if and only if $\dim S>0$,
in which case $\partial S\ne\emptyset$.
\begin{lem}
\label{lem:footpoint} Assume that $X$ has $\curv\geq1$, nonempty
boundary, and $\rad X=\pitwo$. If $E\neq\partial X$, $x\in S$,
and $q\in\partial X$ is closest to $x$, then $q\in S$. In particular,
$\partial S\neq\emptyset$. 
\end{lem}

\begin{proof}
By the choice of $q$, $\angle\left(\overrightarrow{qx},\overrightarrow{qe}\right)\leq\frac{\pi}{2}$
for any $\overrightarrow{qe}$, $e\in E$. As $\abs{xe}=\frac{\pi}{2}$
and $\abs{xq}<\frac{\pi}{2}$, this implies $\abs{eq}=\frac{\pi}{2}$.
Thus $q\in S$. In particular, the closest points in $\partial X$
to the soul $s$ lie in $S$ and hence in $\partial S$. 
\end{proof}
\begin{rem}
This shows that $S\cap\partial X$ is nonempty provided $E\neq\partial X$
but not that $\partial S\subset\partial X$. We will construct examples
(3.6, 3.9, 3.10) where the soul $s\in\partial S$. In particular,
$s$ need not be the soul of $S$. 
\end{rem}

In any case, the distance function to $\partial X$ restricted to
$S$ agrees with the distance function in $S$ to $S\cap\partial X$.
Thus
\begin{prop}
\label{prop:Sboundary}The distance function to $S\cap\partial X$
on $S$ is strictly concave and has its maximum at $s$. In particular,
$S$ is homeomorphic to the cone on $S\cap\partial X$. Moreover,
$s$ is the soul of $S$ if and only if $S\cap\partial X=\partial S$,
and if not $s\in\partial S$. 
\end{prop}

\begin{rem}
When $S\cap\partial X\neq\partial S$, then $\partial S=(S\cap\partial X)\cup(\partial S\cap\text{int}X)$.
In this case, $S\cap\partial X$ is a \emph{face} of the boundary
$\partial S$, and $s$ is the soul point of $S$ relative to this
face. 
\end{rem}

We have one more simple general observation about $S$.
\begin{prop}
\label{prop:dimS=00003D0}When $\dim S=0$, then $E=\partial X$ and
$X=C_{1}E=C_{1}\partial X$. 
\end{prop}

\begin{proof}
Proposition \ref{prop:dual} shows that any gradient curve for $E$
ends in $S$. If we start in a direction of $\partial X$, then the
gradient curve will stay in the extremal set $\partial X$. Therefore,
it ends in a point $\partial X\cap S$. This is clearly not possible
when $S=\left\{ s\right\} $, so it follows that $E=\partial X$. 

Next we give the details of the rigidity statement along the same
lines as in the proof of theorem \ref{thm:rad-rigid-generic}. Consider
$f_{\partial}\left(x\right)=\sin\abs{x\partial X}$ and $f_{1}\left(x\right)=1-\cos\abs{xs}$.
These functions satisfy $\ddot{f}_{\partial}\left(x\right)\leq-\sin\abs{x\partial X}$
(see \cite[theorem 3.3.1]{Pe}) and $\ddot{f}_{1}\left(x\right)\leq\cos\abs{xs}$.
Since every point on the boundary is at distance $\pitwo$ from $s$,
it follows that $\abs{x\partial X}+\abs{xs}\geq\pitwo$. Hence, 
\begin{eqnarray*}
\ddot{f}_{\partial}\left(x\right)+\ddot{f}_{1}\left(x\right) & \leq & -\sin\abs{x\partial X}+\cos\abs{xs}\\
 & \leq & -\sin\abs{x\partial X}+\cos\left(\pitwo-\abs{x\partial X}\right)\\
 & = & 0.
\end{eqnarray*}
On the other hand 
\begin{eqnarray*}
f_{\partial}\left(x\right)+f_{1}\left(x\right) & = & \sin\abs{x\partial X}+1-\cos\abs{xs}\\
 & \geq & \sin\abs{x\partial X}+1-\cos\left(\pitwo-\abs{x\partial X}\right)\\
 & = & 1
\end{eqnarray*}
with equality holding for any $x$ that lies on a geodesic from $s$
to $\partial X$. The minimum principle then shows that $f_{\partial}+f_{1}=1$
on all of $X$. This shows that the gradient exponential map $\mathrm{gexp}_{s}\left(1;\cdot\right):C_{1}\left(S_{s}X\right)\rightarrow X$
is an isometry.
\end{proof}
\medskip{}

\begin{prop}
\label{prop:E}Let $X$ be an Alexandrov space with $\curv\ge1$,
nonempty boundary, and $\rad=\pi/2$. It follows that 
\begin{enumerate}
\item the gradient curves for $r\left(x\right)=\abs{x\partial X}$ that
start in $E$ are minimal geodesics from $E$ to $s$, 
\item $\rad E\geq\frac{\pi}{2}$. 
\end{enumerate}
\end{prop}

\begin{proof}
Let $c:\left[0,L\right]\rightarrow X$ be a gradient curve for $r$
reparametrized by arclength with $c\left(0\right)\in\partial X$ and
$c\left(L\right)=s$. We will use comparison along this gradient curve
as in \cite[lemma 2.1.3]{Pe}. To that end consider $f\left(t\right)=\sin\left(r\circ c\right)$
and note that $c$ is clearly also a gradient curve for $f$. This
implies $\ddot{f}+f\leq0$ in the support sense. Since $f\left(0\right)=0$
this shows that $f\left(t\right)\leq\dot{f}\left(0\right)\sin t$.
As $f\geq0$, this implies that $L\leq\pi$. Define $g\left(\tau\right)=f\left(L-\tau\right)$
and note that also $\ddot{g}+g\leq0$ in the support sense. This time
\[
g\left(\tau\right)\leq g\left(0\right)\cos\tau+\dot{g}\left(0\right)\sin\tau.
\]
At $\tau=L$ this becomes 
\[
0\leq g\left(0\right)\cos\left(L\right)+\dot{g}\left(0\right)\sin\left(L\right),
\]
where $g\left(0\right)>0$ and $\dot{g}\left(0\right)\leq0$ as $r$,
$f$, and hence $g$ are maximal at $s$. Since $L\leq\pi$, this
forces $L\leq\frac{\pi}{2}$. On the other hand we always have $L\geq\abs{s\,c\left(0\right)}$
so when $c\left(0\right)=e\in E$ this shows that $L=\frac{\pi}{2}$
and that the gradient curve must be a minimal geodesic from $e$ to
$s$. This proves (1).

For (2) we use that $E$ consists of all points at distance $\frac{\pi}{2}$
from $s$. Moreover, by (1) of Proposition \ref{prop:dual} intrinsic
distances in $E$ are the same as the extrinsic distances in $X$.
Assume that $B\left(E,\epsilon\right)\subset B\left(e,\frac{\pi}{2}\right)$.
In (1) we saw that there is a minimal geodesic $c:\left[0,\frac{\pi}{2}\right]\rightarrow X$
from $s$ to $e$ which is a reparametrized gradient curve for $r$.
We first claim that $\overrightarrow{es}=\dot{c}^{-}\left(\frac{\pi}{2}\right)\in S_{e}X$
is the soul. In fact, by first variation, $d_{e}r\left(\xi\right)$,
$\xi\in S_{e}X$ is maximal when $\xi$ is the soul of $S_{e}X$ (see
\cite[definition 1.3.2 ]{Pe}). Thus, $\angle\left(w,\overrightarrow{es}\right)\leq\frac{\pi}{2}$
for all $w\in S_{e}X$ and by Toponogov comparison $B\left(E,\epsilon\right)\subset B\left(c\left(t\right),\frac{\pi}{2}\right)$
for all $t>0$. However, also $X-B\left(E,\epsilon\right)\subset B\left(c\left(t\right),\frac{\pi}{2}\right)$
for sufficiently small $t$. This shows that $\rad X<\frac{\pi}{2}$. 
\end{proof}
As an immediate consequence we have
\begin{cor}
\label{prop:dimE=00003D0}When $\dim E=0$, then $X=\Sigma_{1}S$. 
\end{cor}

\begin{proof}
Since $E$ is $\pi$-convex and has $\rad\geq\frac{\pi}{2}$ it follows
that $E=\left\{ 0,\pi\right\} $. Thus $\diam X=\diam E=\pi$ and
the result follows. 
\end{proof}
Define $E^{\prime}\subset S_{s}X$ as the directions $\overrightarrow{se}$,
$e\in E$, that correspond to the gradient curves for $r$ as in part
(1) of proposition \ref{prop:E}. Note that we have not excluded the
possibility that there might be other minimal geodesics from $s$
to points in $E$ that do not correspond to such gradient curves.
\begin{lem}
\label{lem:gradexp} $E^{\prime}$ and $E$ are isometric via the
spherical gradient exponential map. Consequently, $E^{\prime}\subset S_{s}X$
is closed, convex, and $\rad E^{\prime}\geq\frac{\pi}{2}$. 
\end{lem}

\begin{proof}
The goal is to show that the reparametrized gradient curves for $r$
that start in $E$ and end in $s$ form rigid constant curvature 1
triangles along minimal geodesics in $E$. As these curves are minimal
geodesics it follows that they are also gradient curves for the distance
to $s$. This will show that the gradient exponential map yields an
isometry from $E^{\prime}$ to $E$.

The proof uses the parallel translation construction from \cite{PePT}.
Consider a minimal geodesic $c\left(\tau\right):\left[0,b\right]\rightarrow E$
and fix a gradient curve reparametrized by arclength from $c\left(\tau_{0}\right)$
to $s$ for some $\tau_{0}\in\left(0,b\right)$. As $c$ is at constant
distance $\frac{\pi}{2}$ from $s$, there must be a rigid geodesic
triangle $c\left(\tau,t\right):\left[0,b\right]\times\left[0,\frac{\pi}{2}\right]\rightarrow X$,
where $t\mapsto c\left(\tau,t\right)$ is a unit speed geodesic from
$c\left(\tau\right)$ to $s$ and $t\mapsto c\left(\tau_{0},t\right)$
is the given reparametrized gradient curve for $r$ (see \cite{GM}).
There are choices involved in the construction of parallel translation,
however, inside the rigid triangle the field $\frac{dc}{dt}\left(\tau,0\right)$
is intrinsically parallel. Therefore, regardless of other choices,
we can always start by declaring that the constructions in \cite{PePT}
map $\frac{dc}{dt}\left(\tau_{1},0\right)$ to $\frac{dc}{dt}\left(\tau_{2},0\right)$
for $\tau_{1,2}\in\left(0,b\right)$. Thus we can assume that parallel
translations $P_{\tau}:S_{c\left(\tau_{0}\right)}X\rightarrow S_{c\left(\tau\right)}X$
preserve the rigid triangle for all $\tau\in\left(0,b\right)$. As
$P_{\tau}$ is also an isometry it follows that it must map the soul
$\frac{dc}{dt}\left(\tau_{0},0\right)\in S_{c\left(\tau_{0}\right)}X$
to the soul of $S_{c\left(\tau\right)}X$. This shows that the direction
$\frac{dc}{dt}\left(\tau,0\right)\in S_{c\left(\tau\right)}X$ is
the soul for any $\tau\in\left(0,b\right)$ and that $t\mapsto c\left(\tau,t\right)$
is in fact the gradient curve from $c\left(\tau\right)$ to $s$.
By continuity of gradient curves this will also be the case for $\tau=0,b$.
This proves the claim. 
\end{proof}
This lemma immediately tells us that 
\[
\diam E^{\prime}=\diam E\leq\diam X.
\]

This will be crucial for the proof of the Interior Regularity and
Maximal Volume Theorems.

Before turning to the proofs of the theorems in the introduction we
present examples that illustrate various phenomena discussed up till
now.

\section{Examples}

The first result gives an easy way of checking that that the radius
of the examples below indeed are $\pitwo$. 
\begin{prop}
\label{prop:rad=00003Dpi/2}Assume that $A,B\subset X$ are convex
subsets such that 
\[
\abs{ab}=\frac{\pi}{2}\textrm{ and }\abs{Ax}+\abs{xB}=\frac{\pi}{2}
\]
for all $a\in A,\,b\in B$, and $x\in X$. If $\rad B\leq\frac{\pi}{2}$
and $\rad A\geq\frac{\pi}{2}$, then $\rad X=\frac{\pi}{2}$. 
\end{prop}

\begin{proof}
Clearly $X\subset\bar{B}\left(b,\frac{\pi}{2}\right)$ for any $b\in B$.
We claim that for each $x\in X$, there exists $a\in A$ with $\abs{xa}=\frac{\pi}{2}$.
By assumption there is a geodesic $c:\left[0,\frac{\pi}{2}\right]\rightarrow X$
such that $c\left(0\right)\in A$, $c\left(\frac{\pi}{2}\right)\in B$,
and $x=c\left(t\right)$ for some $t$. Let $a\in A$ be chosen so
that $\abs{a\,c\left(0\right)}=\frac{\pi}{2}$. Since $A$ is at constant
distance $\frac{\pi}{2}$ from $c\left(\frac{\pi}{2}\right)$ we obtain
totally geodesic triangles that contain $c$ and $a$. In each of
these triangles the intrinsic distance from $x$ to $a$ is $\frac{\pi}{2}$.
These triangles are uniquely determined by a minimal geodesic from
$x$ to an interior point on a geodesic from $c\left(0\right)$ to
$a$ (see \cite{GM}). If we select the point very close to $a$,
then we obtain a contradiction provided $\abs{xa}<\frac{\pi}{2}$. 
\end{proof}
\begin{rem}
Note that any spherical join $A*B$ with $\rad B\leq\frac{\pi}{2}$
and $\rad A\geq\frac{\pi}{2}$ satisfies the conditions of the proposition.
Below we give examples which are quotients of such joins but not themselves
spherical joins. Moreover, in the special case of a join the conditions
are necessary in the sense that when $\rad A,\rad B<\pitwo$, then
$\rad A*B<\pitwo$. The sphere of radius $\frac{1}{2}$, $X=\bbS^{n}\left(\frac{1}{2}\right)$,
with $A,B$ being a pair of antipodal points is an example that has
$\rad X=\pitwo$; satisfies the first condition in proposition \ref{prop:rad=00003Dpi/2};
while $\rad A=\rad B=0$.
\end{rem}

We start with a simple example to show that one cannot always expect
to obtain a submetry $X\rightarrow\left[0,\frac{\pi}{2}\right]$ when
$\curv\geq1$ and $\diam=\frac{\pi}{2}$. 
\begin{example}
Consider an ellipse $X$ given by 
\[
\frac{x^{2}}{a^{2}}+\frac{y^{2}}{b^{2}}+\frac{z^{2}}{c^{2}}=1,
\]
where $1>a>b>c>0$. The smallest curvature is obtained at $z=\pm c$
and is given by $\frac{c^{2}}{a^{2}b^{2}}$. We select $c=\frac{1}{4}$
and $b=\frac{1}{3}$. In order to have curvature $\geq1$ we then
need $a<\frac{3}{4}$. The diameter of the ellipse is the distance
between the two points $x=\pm a$. This distance is half the perimeter
of the ellipse 
\[
\frac{x^{2}}{a^{2}}+\frac{z^{2}}{c^{2}}=1
\]
and can be estimated by 
\[
\frac{\pi}{2}\left(a+c\right)<\diam X<2\left(a+c\right).
\]
When $a=\frac{3}{4}$ we obtain $\diam X>\frac{\pi}{2}$, while $a=\frac{1}{3}$
gives $\diam X<\frac{7}{6}<\frac{\pi}{2}$. So for some $a\in\left(\frac{1}{3},\frac{4}{3}\right)$
we obtain an ellipse with $\curv\geq1$ and $\diam=\frac{\pi}{2}$.
This gives an example where there is no submetry $X\rightarrow\left[0,\frac{\pi}{2}\right]$. 
\end{example}

The remainder of the section contains various examples of Alexandrov
spaces with nonempty boundary, $\curv\geq1$, and $\rad=\frac{\pi}{2}$. 
\begin{example}
The most basic construction is a spherical join $S*E$, where $S$
has $\curv\geq1$, $\partial S\neq\emptyset$, $\rad S<\frac{\pi}{2}$;
and $E$ has $\curv\geq1$, $\partial E=\emptyset$, $\rad E\geq\frac{\pi}{2}$.
To see how we might obtain such a decomposition consider a spherical
join $\left[0,\pi\right]*\left[0,\pi\right]$ where both factors have
boundary and radius $\frac{\pi}{2}$. This space has radius $\frac{\pi}{2}$
and can be rewritten as follows: 
\begin{eqnarray*}
\left[0,\pi\right]*\left[0,\pi\right] & = & \left[0,\pi\right]*\left(\left\{ \frac{\pi}{2}\right\} *\left\{ 0,\pi\right\} \right)\\
 & = & \left(\left[0,\pi\right]*\left\{ \frac{\pi}{2}\right\} \right)*\left\{ 0,\pi\right\} \\
 & = & \left(\left[0,\frac{\pi}{2}\right]*\left\{ 0,\pi\right\} \right)*\left\{ 0,\pi\right\} \\
 & = & \left[0,\frac{\pi}{2}\right]*\left(\left\{ 0,\pi\right\} *\left\{ 0,\pi\right\} \right)\\
 & = & \left[0,\frac{\pi}{2}\right]*S^{1}\left(1\right).
\end{eqnarray*}
Similarly, we have 
\begin{eqnarray*}
\left[0,\alpha\right]*\left[0,\pi\right] & = & \left(\left[0,\alpha\right]*\left\{ \frac{\pi}{2}\right\} \right)*\left\{ 0,\pi\right\} \\
 & = & \Sigma_{1}\left(\left[0,\alpha\right]*\left\{ \frac{\pi}{2}\right\} \right).
\end{eqnarray*}
\end{example}

\begin{example}
Assume we have an example $X=S*E$ as above and a compact group $G$
that acts isometrically and effectively on $S$ and $E$. This action
is naturally extended to $S*E$ in such a way that it preserves the
slices $S\times\left\{ t\right\} \times E$ at constant distance from
$S$ and $E$. On these slices it is the diagonal action by $G$ on
$S\times E$. This leads to a new Alexandrov space $X/G$. Note that
$G$ preserves $\partial X$ and $\partial S$ and consequently also
the common soul of both spaces. In particular, $E/G\subset\left(\partial X\right)/G$
is at maximal distance $\frac{\pi}{2}$ from the soul and $E/G$ and
$S/G$ are dual sets in $X/G$. It follows from proposition \ref{prop:rad=00003Dpi/2}
that $\rad X/G\geq\frac{\pi}{2}$ provided $\rad E/G\geq\frac{\pi}{2}$.
Topologically, $S*E$ is a cone over $\partial X=\left(\partial S\right)*E$,
where the action fixes the soul and preserves the boundary. The quotient
is likewise a topological cone over $\left(\partial X\right)/G$. 
\end{example}

Below we offer some concrete examples of this construction. 
\begin{example}[Projective Lenses]
Consider $E=\left\{ 0,\pi\right\} $ and $S=\left[-\alpha,\alpha\right],\alpha<\pitwo$.
Let $G=\mathbb{Z}_{2}$ be the natural reflection on both spaces.
Note that $S*E=\Sigma_{1}\left[-\alpha,\alpha\right]=L_{2\alpha}^{2}$
looks topologically like a hemisphere. The action fixes $0\in\left[-\alpha,\alpha\right]$
and acts like the antipodal map on $\partial\left(S*E\right)=\Sigma_{1}\left\{ -\alpha,\alpha\right\} =\bbS^{1}\left(1\right)$.
The quotient looks topologically like a cone with vertex $0$. The
boundary has one point at distance $\frac{\pi}{2}$ from $0$. The
issue is that $\rad\left(E/G\right)<\frac{\pi}{2}$ and $\rad\left(X/G\right)<\frac{\pi}{2}$. 

More generally one can consider the $\mathbb{Z}_{2}$ quotient of
the Alexandrov lens $L_{2\alpha}^{n}=S^{n-2}\left(1\right)*\left[-\alpha,\alpha\right]$
(cf. also \cite{GL2}). When $n>2$, this space will have $\rad=\frac{\pi}{2}$
and boundary isometric to $\bbR\bbP^{n-1}$. 
\end{example}

\begin{example}[Edge with Nonempty Boundary]
Consider $S=\left(\Sigma_{1}\left[-\alpha,\alpha\right]\right)/\mathbb{Z}_{2}$
as above and define $X=\bbS^{1}\left(r\right)*S$, where $r\in\left[\frac{1}{2},1\right]$.
The soul of $X$ is the soul of $S$ and $E=C_{1}\bbS^{1}\left(r\right)$.
This is an example where $E$ has nonempty boundary.
\end{example}

Further examples that indicate the complexities in trying to classify
spaces with maximal radius can be obtained as follows: 
\begin{example}[Higher Dimensional Spines]
Select $E=S^{1}\left(1\right)$ and $S=B\left(p,r\right)\subset S^{2}\left(1\right)$,
where $r<\frac{\pi}{2}$. Let $G=\mathbb{Z}_{2}$ be a rotation by
$\pi$ on both $S$ and $E$. This gives a 4-dimensional example where
the boundary is homeomorphic to $\bbR\bbP^{3}$. The same can be done
with $E=S^{n}\left(1\right)$ and $S=B\left(p,r\right)\subset S^{m}\left(1\right)$,
and $G=\mathbb{Z}_{2}$ the antipodal map on both $S$ and $E$, giving
an $n+m+1$-dimensional example with boundary homeomorphic to $\bbR\bbP^{n+m}$.
\end{example}

\begin{example}[Spines with Soul on Boundary]
Select $E=S^{1}\left(1\right)$ and $S=B\left(p,r\right)\subset S^{2}\left(1\right)$,
where $r<\frac{\pi}{2}$. Let $G=\mathbb{Z}_{2}$ be a rotation by
$\pi$ on $E$ and a reflection on $S$. This gives a 4-dimensional
example where the boundary of $S/\mathbb{Z}_{2}$ is connected and
contains the soul. The boundary is homeomorphic to a suspension $\Sigma\bbR\bbP^{2}$.
As in the previous example we can choose $E=S^{n}\left(1\right)$
and $S=B\left(p,r\right)\subset S^{m}\left(1\right)$, with $G=\mathbb{Z}_{2}$
action on $E$ as the antipodal map and on $S$ as a reflection (or
any other isometric involution). The resulting example is $n+m+1$-dimensional
with boundary homeomorphic to $\Sigma^{m-1}\bbR\bbP^{n+1}$.

In both of these examples the key is that $S$ is an Alexandrov space
with curvature at least 1, non-empty boundary and radius $r<\frac{\pi}{2}$,
and with an isometric involution.
\end{example}

In corollary \ref{cor:leq4} we will show that above examples exhaust
all the possibilities in dimensions $\leq4$, while the next examples
shows that one can have more complex behavior in dimensions $\geq5$.
\begin{example}[Dual Pairs of Nonmaximal Dimension]
Select $E=S^{3}\left(1\right)$ and $S=B\left(p,r\right)\subset S^{2}\left(1\right)$,
where $r<\frac{\pi}{2}$. Let $G=S^{1}$ be the Hopf action on $E$
and rotation around $p$ on $S$. This gives a 5-dimensional example
that is not a finite quotient of a spherical join and with boundary
homeomorphic to $\bbC\bbP^{2}$. 
\end{example}

\section{Interior regularity and Lytchak's problem}

We need the following result for convex subsets of the standard sphere.
\begin{prop}
If $A\subset\mathbb{S}^{n-1}\left(1\right)$ is closed, convex and
has $\rad A\geq\frac{\pi}{2}$, then $\diam A=\pi$. 
\end{prop}

\begin{proof}
In case $\partial A=\emptyset$, the radius condition is redundant
and $A$ is totally geodesic unit sphere or two antipodal points.

In general note that $\rad A<\frac{\pi}{2}$ is equivalent to $A$
lying in an open hemisphere. Since $A$ is convex we have that it
lies in a closed hemisphere $H$. In case $\rad\left(A\cap\partial H\right)<\frac{\pi}{2}$
it follows that we can move $H$ so that $A$ lies in an open hemisphere.
In this way we obtain a convex subset $A\cap\partial H\subset\partial H$
inside a lower dimensional sphere with radius $\geq\frac{\pi}{2}$.
This ultimately reduces the problem to the trivial case: $A\subset\mathbb{S}^{0}\left(1\right)$,
where the radius condition forces $A=\mathbb{S}^{0}\left(1\right)$. 
\end{proof}
\begin{cor}
If $S_{s}X=\mathbb{S}^{n-1}\left(1\right)$, then $\diam E^{\prime}=\pi$. 
\end{cor}

\begin{proof}
As $E^{\prime}\subset\mathbb{S}^{n-1}\left(1\right)$ is closed, $\rad E^{\prime}\geq\frac{\pi}{2}$,
and convex, the result follows from the previous proposition. 
\end{proof}
We are now ready to complete the proof of the Interior Regularity
Theorem from the introduction. 
\begin{thm}
\label{thm:reg_soul}Let $X$ be an $n$-dimensional Alexandrov space
with $\curv\geq1$ and $\partial X\ne\emptyset$. If $\rad X=\frac{\pi}{2}$
and the soul of $X$ is a regular point, then $X$ is isometric to
a spherical join $\mathbb{S}^{k}\left(1\right)*S$ , where $S$ is
an $\left(n-k-1\right)$-dimensional Alexandrov space with $\curv\geq1$,
$\partial S\ne\emptyset$, and $\rad S<\frac{\pi}{2}$. 
\end{thm}

\begin{proof}
The assumption that the soul is regular means that $S_{s}X=\mathbb{S}^{n-1}\left(1\right)$.
From the preceding corollary we know that $X=\Sigma_{1}S_{p}X$ for
some $p\in E\subset\partial X$. Further note that the soul $s$ lies
in the slice $\left\{ \frac{\pi}{2}\right\} \times S_{p}X$ and also
corresponds to the soul of $S_{p}X$. Moreover as 
\[
\mathbb{S}^{n-1}\left(1\right)=S_{s}X=S_{s}\Sigma_{1}S_{p}X=\Sigma_{1}S_{s}S_{p}X
\]
it follows that $S_{p}X$ also has the property that its soul is a
regular point. When $\rad S_{p}X<\frac{\pi}{2}$ we have obtained
the desired decomposition, otherwise the construction can be iterated
until one reaches the desired decomposition. 
\end{proof}
From this we deduce the rigidity part of Lytchak's problem by first
observing the following reformulation of the results in \cite[3.3.5]{Pe}. 
\begin{prop}
\label{prop:Lytchak}Let $X$ be an $n$-dimensional Alexandrov space
with $\curv\geq1$, $\partial X\neq\emptyset$, and $s$ the soul
of $X$. If $\vol_{n-1}\partial X=\vol_{n-1}\mathbb{S}^{n-1}\left(1\right)$,
then $\rad X=\frac{\pi}{2}$, $X=\bar{B}\left(s,\frac{\pi}{2}\right)$,
\[
\gexp_{s}\left(1;S_{s}X\right)=\partial X=\partial\bar{B}\left(s,\frac{\pi}{2}\right),
\]
and $S_{s}X$ is isometric to $\mathbb{S}^{n-1}\left(1\right)$. 
\end{prop}

\begin{proof}
It follows from Petrunin's solution to Lytchak's problem \cite[3.3.5]{Pe}
that $\rad X\leq\frac{\pi}{2}$. Similarly, if $r\leq\frac{\pi}{2}$
and $X=\bar{B}\left(p,r\right)$ for some $p\in M$, then $\partial X\subset\gexp_{p}\left(1;\partial\bar{B}\left(p,r\right)\right)$.
In particular, as $\gexp_{p}\left(1;\cdot\right)$ is distance nonincreasing:
\begin{eqnarray*}
\vol_{n-1}\left(\partial X\right) & \leq & \vol_{n-1}\left(\gexp_{p}\left(1;\partial\bar{B}\left(o_{p},r\right)\right)\right)\\
 & \leq & \vol_{n-1}\left(\partial\bar{B}\left(o_{p},r\right)\right)\\
 & \leq & \vol_{n-1}\mathbb{S}^{n-1}\left(1\right).
\end{eqnarray*}
Here equality can only hold when $r=\frac{\pi}{2}$ and in that case
we can use $X=\bar{B}\left(s,\frac{\pi}{2}\right)$. Moreover, $\partial X=\gexp_{s}\left(1;\partial\bar{B}\left(s,\frac{\pi}{2}\right)\right)$
as otherwise $\partial\bar{B}\left(o_{s},\frac{\pi}{2}\right)\cap\gexp_{s}^{-1}\left(\mathrm{int}X\right)$
is a nonempty open set and that forces $\vol_{n-1}\left(\partial X\right)<\vol_{n-1}\left(\gexp_{s}\left(1;\partial\bar{B}\left(o_{s},\frac{\pi}{2}\right)\right)\right)$.
Finally, we know that $S_{s}X=\partial\bar{B}\left(o_{s},\frac{\pi}{2}\right)$
also has maximal volume, showing that $S_{s}X$ is isometric to $S^{n-1}\left(1\right)$. 
\end{proof}
From \ref{thm:reg_soul} it follows that when $\vol_{n-1}\partial X=\vol_{n-1}\mathbb{S}^{n-1}\left(1\right)$,
then $\partial X$ is isometric to $\mathbb{S}^{k}\left(1\right)*\partial S$
, where $S$ is an $\left(n-k-1\right)$-dimensional Alexandrov space
with $\curv\geq1$. This is isometric to $\mathbb{S}^{n-1}\left(1\right)$
if and only if $S=[0,\alpha]$. Consequently, we have answered Lytchak's
problem as in the Maximal Volume Theorem from the introduction.
\begin{cor}
\label{cor:MaxVol} Let $X$ be an $n$-dimensional Alexandrov space
with $\curv\geq1$ and nonempty boundary. If $\vol\partial X=\vol\mathbb{S}^{n-1}\left(1\right)$,
then $X$ is isometric to $L_{\alpha}^{n}$ for some $0<\alpha\leq\pi$. 
\end{cor}

These results can be used to complement (if not complete) the main
theorem in \cite{GL2}.
\begin{prop}
If $X^{n}$ is an Alexandrov space with $\curv\geq1$, $\rad X=\pitwo$,
and boundary $\partial X=M$ that is a Riemannian manifold and a topological
sphere with $\sec\ge1$, then $M=\mathbb{S}^{n-1}\left(1\right)$.
Consequently, $X$ is an Alexandrov lens.
\end{prop}

\begin{proof}
Suppose $\partial E\ne\emptyset$. Since $E$ is a $\pi$-convex subset
of $M$ with $\rad E\geq\pitwo$ it must have radius $\pitwo$ and
by the Regularity Theorem \ref{thm:reg_soul} $E$ becomes a join
with a unit sphere. In particular,
\[
\pi=\diam E\leq\diam M
\]
and by Toponogovs maximal diameter theorem $M=\mathbb{S}^{n-1}\left(1\right)$.

It remains to consider the case where $E$ is a smooth totally geodesic
submanifold of $M$ with $\rad\ge\pitwo$. Now $M\cap S$ is a $\pi$-convex
subset of $M$ and dual to $E$ in the sense of \cite{GG}. The arguments
in \cite{GG} show that $M\cap S$ is a smooth totally geodesic submanifold
of $M$ without boundary. Since $M$ is topologically a sphere with
nontrivial dual submanifolds it follows again from \cite{GG} that
all points in $M-(E\cup(S\cap M))$ lie on a unique minimal geodesic
of length $\pitwo$ from $E$ to $S\cap M$. Further, for each $x\in S\cap M$
the corresponding map from the normal sphere to $S\cap M$ at $x$
to $E$ is a Riemannian submersion (and likewise for points in $E$).
The classification of Riemannian submersions from spheres (cf. \cite{GG,Wi})
and the fact that $M$ is a topological sphere implies that $M=\mathbb{S}^{n-1}\left(1\right)$. 
\end{proof}

\section{Rigidity from Topology }

In this section we discuss several more general results for Alexandrov
spaces $X^{n}$ with $\curv\ge1$, $\partial X\ne\emptyset$, and
maximal radius $\rad=\pitwo$. These include the generalizations to
the inner regularity theorem mentioned in the introduction and lead
to a classification in dimensions $\leq4$.

For the purposes of this section we shall need an improved dual set
decomposition of $X$.
\begin{prop}
For an Alexandrov space $X^{n}$ with $\curv\ge1$, $\partial X\neq\emptyset$,
and $\rad X=\pitwo$ there exists a dual space decomposition $\hat{E},\hat{S}\subset X$
with the properties that $\hat{E}\subset E$, $\partial\hat{E}=\emptyset$,
$\rad\hat{E}\geq\pitwo$, $S\subset\hat{S}$, and $\partial\hat{S}\neq\emptyset$. 
\end{prop}

\begin{proof}
When $\partial E=\emptyset$ there is nothing to prove. Otherwise
we have $\rad E=\pitwo$ and $E$ itself admits a dual space decomposition
$E_{1}\subset\partial E$, $S_{1/2}\subset E$. Define $S_{1}=\left\{ x\in X\mid\abs{xE_{1}}\geq\pitwo\right\} $.
Note that $S,S_{1/2}\subset S_{1}$. Thus any point at distance $\pitwo$
from $S_{1}$ must lie in $E$ and hence also in $E_{1}$. This shows
that $E_{1},S_{1}\subset X$ are dual to each other. By construction
$\rad E_{1}\geq\pitwo$ and $S_{1}\cap\partial X\neq\emptyset$. We
can now continue this procedure until the desired decomposition is
reached. 
\end{proof}
\begin{rem}
Note that we haven't claimed $\rad\hat{S}<\pitwo$.
\end{rem}

Our rigidity results depend on the following version of Lefschetz
duality.
\begin{thm}
Assume $Z$ is a compact connected ANR that is a $\bbZ_{2}$-homology
sphere, and $A,B\subset Z$ are disjoint, compact, connected, and
ANR. If $B\subset Z-A$ is a deformation retract and $Z-A$ is a topological
$n$-manifold, then
\[
H_{q}\left(B;R\right)\simeq H^{n-1-q}\left(A;R\right),\,q=1,...,n-2.
\]
\end{thm}

\begin{proof}
It follows from Lefschetz duality and the fact that $A$ and $Z$
are ANRs that for all $q$
\[
H_{q}\left(Z-A;\bbZ_{2}\right)\simeq H^{n-q}\left(Z,A;\bbZ_{2}\right).
\]
The fact that $H^{p}\left(Z;\bbZ_{2}\right)=0$ for $p=1,...,n-1$
shows, via the long exact sequence for relative cohomology, that
\[
H^{n-q}\left(Z,A;\bbZ_{2}\right)\simeq H^{n-1-q}\left(A;\bbZ_{2}\right),q=1,...,n-2.
\]
The fact that $B\subset Z-A$ is a deformation retract then implies
the claim.
\end{proof}
We require some extra notation. For a convex subset $A\subset X$
of an Alexandrov space we define the normal space at $a\in A$ as
\[
N_{a}A=\left\{ v\in S_{a}X\mid\angle\left(v,w\right)\geq\pitwo\textrm{ for all }w\in S_{a}A\right\} .
\]
By first variation any unit speed geodesic that starts in $a$ and
minimizes the distance to $A$ has initial velocity that lies in $N_{a}A$.
\begin{lem}
\label{lem:dual_max_dim}Assume $Z=\bar{Z}/H$ is an Alexandrov space
with $\curv\geq1$ and $\partial Z=\emptyset$, where $H$ is a finite
group of isometries, $\bar{Z}$ is a closed topological $n$-manifold
that is a $\bbZ_{2}$-homology sphere, and an Alexandrov space with
$\curv\geq1$. If $A,B\subset Z$ form a dual pair and $\partial A=\emptyset$,
then there exists $x\in B,\,y\in A$ and finite group, $G$, that
acts effectively and isometrically on both $N_{x}B$ and $N_{y}A$,
such that $Z=\left(N_{x}B*N_{y}A\right)/G$.
\end{lem}

\begin{proof}
The goal is to prove that
\[
\dim A+\dim B=n-1
\]
and that no points in $B$ have distance $>\pitwo$ to $A$ and vice
versa. This allows us to use \cite[theorem A]{RW} when $\partial B=\emptyset$
and otherwise \cite[theorem B]{RW} to reach the conclusion of the
lemma.

We lift the situation to $\bar{A},\bar{B}\subset\bar{Z}$. In case
$\bar{A}$ (or $\bar{B}$) is not connected the components must be
distance $\pi$ apart as $\bar{A}$ is $\pi$-convex. This forces
$\bar{A}$ to consist of two points and $\bar{Z}$ to be a suspension
with $\bar{B}=S_{a}X$, $a\in\bar{A}$. So we can assume that $\dim\bar{A}=p>0$.
Since $\bar{A},\bar{B}$ form a dual pair it follows that each of
these sets contains the set of critical points for the distance function
to the other set. The gradient flow then shows that $\bar{Z}-\bar{A}$
deformation retracts to $\bar{B}$. From $\partial\bar{A}=\emptyset$
we conclude that $H^{p}\left(\bar{A},\bbZ_{2}\right)=\bbZ_{2}$. By
Alexander duality we can then conclude that $H_{q}\left(\bar{B},\bbZ_{2}\right)=\bbZ_{2}$
for $p=n-1-q$. This shows that 
\[
\dim\bar{A}+\dim\bar{B}\geq n-1.
\]
Since $\bar{B}\subset\bar{Z}$ is convex it follows that $\partial\bar{B}=\emptyset$
as it would otherwise be contractible. Frankel's theorem for Alexandrov
spaces (see \cite{PePT}) then shows that
\[
\dim\bar{A}+\dim\bar{B}\leq n-1.
\]
 Moreover, points in $\bar{B}$ cannot have distance $>\pitwo$ from
$\bar{A}$ and vice versa. This finishes the proof.
\end{proof}
\begin{rem}
It is in general not possible to conclude that $\bar{Z}$ in lemma
\ref{lem:dual_max_dim} is a join. The icosahedral group $G$ acts
freely on $\bbS^{3}\left(1\right)$ and hence on $\bbS^{3}\left(1\right)*\bbS^{3}\left(1\right)=\bbS^{7}\left(1\right)$.
While the quotient $\bbS^{7}\left(1\right)/G$ is clearly a homology
sphere it is not a join as it is a space form that is not homeomorphic
to a sphere.
\end{rem}

\begin{rem}
\label{rem:diam=00003Dpi/2}Note that from the classification obtained
in \cite{GGG} any positively curved Alexandrov space of dimension
$\leq3$ and empty boundary is of the form $Z=\bar{Z}/H$ where $\bar{Z}$
is homeomorphic to a sphere. If in addition $\diam Z=\pitwo$, then
we obtain two dual sets $A,B\subset Z$. If both of these have boundary,
then $Z$ is topologically a suspension and therefore topologically
a sphere or $\Sigma\bbR\bbP^{2}$. Otherwise we can apply the lemma.
\end{rem}

We can now prove the Topological Regularity Theorem from the introduction. 
\begin{thm}
\label{thm:top}Let $X^{n}$ be an Alexandrov space with $\curv\ge1$,
$\partial X\neq\emptyset$, and $\rad X=\pitwo$. If $\partial X$
is a topological manifold and a $\bbZ_{2}$-homology sphere, then
$D\left(X\right)=\left(X_{1}*X_{2}\right)/G$. Here $G$ is a finite
group acting effectively and isometrically on both $X_{1}$ and $X_{2}$
whose action is extended to the spherical join $X_{1}*X_{2}$.
\end{thm}

\begin{proof}
The idea is simply to apply lemma \ref{lem:dual_max_dim} to the double.
This requires a few minor adjustments. We use the dual decomposition
for $D\left(X\right)$ that consists of $D\left(\hat{S}\right)$ and
the copy of $F\subset D\left(X\right)$ that corresponds to $\hat{E}\subset\partial X$.
Here $D\left(\hat{S}\right)$ will turn out to be the double of $\hat{S}$,
but for now it is simply the preimage of $\hat{S}$. Note that inside
$X$ the gradient flows for $\hat{S}$ and $\hat{E}$ preserve $\partial X$.
Thus we obtain deformation retractions of $D\left(X\right)-D\left(\hat{S}\right)$
to $F$ and $D\left(X\right)-F$ to $D\left(\hat{S}\right)$ relative
to $\partial X$ as in proposition \ref{prop:dual}. Moreover, as
$X$ is homeomorphic to the cone over the boundary it follows that
$D\left(X\right)-D\left(\hat{S}\right)$ is a topological $n$-manifold.
Additionally, $D\left(X\right)$ is a $\bbZ_{2}$-homology sphere
by Meyer-Vietoris. This again shows that
\[
H_{q}\left(F;\bbZ_{2}\right)\simeq H^{n-1-q}\left(D\left(\hat{S}\right);\bbZ_{2}\right),\,q=1,...,n-2.
\]

We can then argue as in the proof of lemma \ref{lem:dual_max_dim}
that $\partial D\left(\hat{S}\right)=\emptyset$ and that we obtain
the desired decomposition for $D\left(X\right)$.
\end{proof}
\begin{rem}
\label{rem:join_boundary}Note that $N_{x}F\subset S_{x}D\left(X\right)$
is the double of $N_{x}\hat{E}\subset S_{x}X$. This shows that $N_{x}F$,
and thus also $N_{y}D\left(\hat{S}\right)*N_{x}F$, come with a natural
reflection whose quotient is $N_{x}\hat{E}$, respectively, $N_{y}D\left(\hat{S}\right)*N_{x}\hat{E}$.
Let $R$ be the natural reflection on both $N_{x}F$ and $D\left(\hat{S}\right)=N_{x}F/G$.
This results in a commutative diagram
\[
\begin{array}{ccc}
N_{x}F & \overset{R}{\rightarrow} & N_{x}F\\
\downarrow &  & \downarrow\\
D\left(\hat{S}\right) & \overset{R}{\rightarrow} & D\left(\hat{S}\right).
\end{array}
\]
In case $G$ commutes with $R$ on $N_{x}F$ it follows that $G$
will also act on $N_{x}\hat{E}$. Consequently, also $X$ becomes
the quotient of a join: $X=\left(N_{x}\hat{E}*N_{y}\hat{S}\right)/G$.
It is not, in general, clear whether $R$ and $G$ will commute. However,
in the case where $G=\left\langle I\right\rangle $, $I^{2}=\mathrm{id}$
this is automatically true. To see this assume $x,Ix\in D\left(Z\right)$
are mapped to $s\in D\left(\hat{S}\right)$. Then $Rx,RIx$ are both
mapped to $Rs$, but the preimage of $Rs$ also consists of the two
points $Rx,IRx$, this shows that $IR=RI$.
\end{rem}

Next we establish the Weak Inner Regularity Theorem from the introduction. 
\begin{thm}
\label{thm:weak_inner}Let $X^{n}$ be an Alexandrov space with $\curv\ge1$,
$\partial X\neq\emptyset$, and $\rad X=\pitwo$. If $\rad S_{s}X>\pitwo$,
then $X=\hat{E}*\hat{S}$, where $\rad\hat{E}>\pitwo$ and $\rad S_{s}\hat{S}>\pitwo$. 
\end{thm}

\begin{proof}
The radius assumption on the space of directions is first used to
see that $X$ is a topological manifold near the soul. Since $X$
is homeomorphic to a cone it follows that it is a topological manifold
with boundary. Finally, the space of directions is homeomorphic to
a sphere and thus $\partial X$ is also homeomorphic to a sphere.
This shows that we can apply theorem \ref{thm:top}. 

Next we can use it in combination with the following fact: If $G$
is a compact group acting by isometries on an Alexandrov space $Y$
with $\curv Y\geq1$, then $\rad\left(Y/G\right)\leq\pitwo$ unless
the action is trivial. Moreover, if the action is free, then $\diam Y/G\leq\pitwo$.
This is obvious when $\dim Y=1$ and follows in general from an induction
argument. To see this, first note that when $\diam\left(Y/G\right)>\pitwo$,
then there are two orbits that are at distance $>\pitwo$ from each
other. Thus one orbit is forced to lie in a $\pi$-convex set that
is at distance $>\pitwo$ from some point. The action must then have
a fixed point $y$ in this set. It now follows from induction that
when the action is nontrivial, then $\rad\left(S_{y}Y/G\right)\leq\pitwo$,
and hence that $\rad\left(Y/G\right)\leq\pitwo$ (see e.g. \cite[lemma 5.2.1]{Pe}).

With notation as in the proof of theorem \ref{thm:top} we show that
$D\left(X\right)=F*D\left(\hat{S}\right)$. Let $s\in D\left(X\right)$
be fixed to be one of the two soul points and note that $S_{s}\hat{S}=S_{s}D\left(\hat{S}\right)$
and $S_{s}X=S_{s}D\left(X\right)$. 

We first show that $\hat{E}^{\prime}=N_{s}\hat{S}$. Observe that
as $\hat{E}^{\prime}\subset N_{s}\hat{S}$ we have:
\[
\dim N_{s}\hat{S}+\dim S_{s}\hat{S}\geq\dim\hat{E}+\dim\hat{S}-1=n-2.
\]
On the other hand as $\partial S_{s}\hat{S}=\emptyset$ it follows
from \cite[theorem A, part (A1)]{RW} that
\[
\dim N_{s}\hat{S}+\dim S_{s}\hat{S}\leq n-2.
\]
Thus $\dim\hat{E}^{\prime}=\dim N_{s}\hat{S}$. In case both spaces
are $0$-dimensional they will both consist of two points at distance
$\pi$ apart. This is because both spaces have $\curv\geq1$ and $\rad\hat{E}^{\prime}\geq\pitwo$
while $N_{s}\hat{S}\subset S_{s}X$ is $\pi$-convex. When both spaces
have dimension $>0$, we know that $N_{s}\hat{S}$ is connected as
it is $\pi$-convex. Since $\dim\hat{E}^{\prime}=\dim N_{s}\hat{S}$
and $\partial\hat{E}^{\prime}=\emptyset$ it follows that $\hat{E}^{\prime}\subset N_{s}\hat{S}$
is an open and closed subset and hence that $\hat{E}^{\prime}=N_{s}\hat{S}$.

This shows that only one ``point'' $\bar{s}\in N_{x}F\subset N_{y}D\left(\hat{S}\right)*N_{x}F$
is mapped to $s\in D\left(\hat{S}\right)\subset D\left(X\right)$
and hence that $\bar{s}$ is a fixed point for the isometric action
of $G$ on $N_{y}D\left(\hat{S}\right)*N_{x}F$. In particular, $G$
preserves $S_{\bar{s}}\left(N_{y}D\left(\hat{S}\right)*N_{x}F\right)$
and $S_{s}X=\left(S_{\bar{s}}\left(N_{y}D\left(\hat{S}\right)*N_{x}F\right)\right)/G$.
However, this can only happen if $G$ acts trivially on $S_{\bar{s}}\left(N_{y}D\left(\hat{S}\right)*N_{x}F\right)$
as $\rad S_{s}X>\pitwo$. In conclusion, $G$ acts trivially on $N_{y}D\left(\hat{S}\right)*N_{x}F$
and $D\left(X\right)$ is a spherical join. This shows that $S_{s}X=\hat{E}^{\prime}*S_{s}\hat{S}$
and $\rad\hat{E}>\pitwo$ and $\rad S_{s}\hat{S}>\pitwo$.

Finally note that the natural reflection on $D\left(X\right)=D\left(\hat{S}\right)*F$
fixes $F$ and has orbit space $X$. Thus also $X=\hat{S}*\hat{E}$.
\end{proof}
\begin{rem}
In the context of this result it is worth noting that the icosahedral
group $I$ acts on $\bbS^{5}\left(1\right)=\bbS^{1}\left(1\right)*\bbS^{3}\left(1\right)$
with a quotient $\bbS^{1}\left(1\right)*\left(\bbS^{3}\left(1\right)/I\right)$
that is a topological sphere (see e.g., \cite{Da}). This quotient
also shows that one can not expect to use general position or transversality
arguments for Alexandrov spaces that are topological manifolds.
\end{rem}

Based on lemma \ref{lem:dual_max_dim} we obtain the following generalization
of theorem \ref{thm:top}.
\begin{cor}
\label{cor:top_branched}Let $X^{n}$ be an Alexandrov space with
$\curv\ge1$, $\partial X\neq\emptyset$, and $\rad X=\pitwo$. If
$S_{s}X$ is a topological manifold that is covered by a $\bbZ_{2}$-homology
sphere or $S_{s}X$ is homeomorphic to $\Sigma\bbR\bbP^{2}=\bbS^{3}/\mathbb{Z}_{2}$,
then $D\left(X\right)=\left(X_{1}*X_{2}\right)/G$. Here $G$ is a
finite group acting effectively and isometrically on both $X_{1}$
and $X_{2}$ whose action is extended to the spherical join $X_{1}*X_{2}$.
\end{cor}

\begin{proof}
The goal is to find a suitable ramified or branched cover of $X$.
In all cases these will in fact be good orbifold covers. Specifically,
for a given Alexandrov space $X$ we seek an Alexandrov space $Y$,
with the same lower curvature bound, and a finite group $G$ acting
by isometries on $Y$ such that $X=Y/G$. The relevant results are
established in \cite[Theorem A]{HS} and \cite[Subsection 2.2]{DGGM}.
The main technical tools for showing that $Y$ has the desired properties
are proven in \cite{Li}. We only need these covers for positively
curved spaces with boundary, i.e., for spaces that are homeomorphic
to cones over the space of directions at the soul $X\simeq C_{1}S_{s}X$.
This means that the cover $Y\simeq C_{1}S_{\bar{s}}Y$ and $\left(S_{\bar{s}}Y\right)/G=S_{s}X$
since $G$ is forced to fix $\bar{s}$. Here $S_{\bar{s}}Y$ is either
a covering space over $S_{s}X$ or a good orbifold cover of $S_{s}X$.

When $S_{s}X$ has a covering space there is only one isolated branch
point and the complement of the soul is convex. This means we can
use exactly the same strategy as in \cite[Subsection 2.2]{DGGM} to
create the Alexandrov space structure on $Y$. When $S_{s}X$ is homeomorphic
to $\Sigma\bbR\bbP^{2}$ it follows that $X$ is not orientable as
it does not have a local orientation at the soul. This means that
we can use the orientation covering as in \cite[Theorem A]{HS} as
$Y$.

The assumptions of the corollary now show that $Y$ exists and is
a topological manifold. We can then apply lemma \ref{lem:dual_max_dim}
to finish the proof.
\end{proof}
\begin{cor}
\label{cor:leq4}When $\dim X\leq4$, then $X$ is either isometric
to a join or a $\bbZ_{2}$ quotient of a join, where $\bbZ_{2}$ acts
effectively on both factors of the join. 
\end{cor}

\begin{proof}
Note that when $\dim X\leq2$, then either $\dim E=0$ or $\dim S=0$.
Thus the result follows from propositions \ref{prop:dimE=00003D0}
and \ref{prop:dimS=00003D0}. Similarly, in higher dimensions we can
assume that both $\hat{S}$ and $\hat{E}$ have positive dimensions.

In case $\dim X=3,4$ we need to use corollary \ref{cor:top_branched}
and remark \ref{rem:join_boundary}. First observe that when $\dim X=3$,
it follows that $S_{s}X$ is homeomorphic to $\bbS^{2}$ or $\bbR\bbP^{2}$,
while when $\dim X=4$ we can use the classification from \cite{GGG}
to conclude that $S_{s}X$ is homeomorphic to a spherical space form
or the suspension over the real projective plane. This places us in
a position where we can use corollary \ref{cor:top_branched}. 

We assume that $D\left(X\right)=X_{1}*X_{2}/G$, with $X_{1}/G=D\left(\hat{S}\right)$
and $X_{2}/G=F$. When $G$ is trivial there is nothing to prove so
we also assume $\abs G\geq2$.

As $\rad F\geq\pitwo$ we note that when $\dim X_{2}=1$ we have $G=\bbZ_{2}$
and acting as the antipodal map since $\partial F=\emptyset$.

Assume that $\dim X_{2}=2$. When $\diam F=\pi$ we obtain a dimension
reduction. 

In case $\diam F\in\left(\pitwo,\pi\right)$ it follows as in the
proof of theorem \ref{thm:weak_inner} that $G$ has two fixed points
$x,y\in X_{2}$ with $\abs{xy}=\diam X_{2}$. We then have that $S_{x,y}F=\left(S_{x,y}X_{2}\right)/G=\bbS^{1}\left(\frac{1}{\abs G}\right)$.
Let $z\in F$ be any point with $\abs{zx}=\abs{zy}<\pitwo$. The radius
assumption shows that there is $\bar{z}\in F$ with $\abs{z\bar{z}}=\pitwo$.
Consider a hinge with vertex at $x$ (or $y$). Since $\diam S_{x,y}F\leq\pitwo$
the angle of the hinge is $\leq\pitwo$. But this leads to a contradiction
as either $\abs{x\bar{z}}<\pitwo$ or $\abs{y\bar{z}}<\pitwo$ which
by Toponogov comparison implies $\abs{z\bar{z}}<\pitwo$. This shows
that this situation is impossible.

Finally we consider the case when $\diam F=\rad F=\pitwo$. This allows
us to obtain dual sets $A,B\subset F$ where, say, $\dim A=0$. By
lemma \ref{lem:dual_max_dim} $F=\left(N_{x}B*N_{y}A\right)/\Gamma$,
where $\Gamma$ acts effectively on $N_{x}B$ and $N_{y}A$. Since
$\dim N_{x}B=0$ this implies that either $N_{x}B$ is a point and
$\Gamma$ is trivial or $N_{x}B$ consists of two points distance
$\pi$ apart and $\Gamma=\bbZ_{2}$. In the former case $F=\left\{ 0\right\} *\bbS^{1}\left(\frac{1}{2}\right)$
which is impossible as $\partial F=\emptyset$. In the latter case
$F=\left(\Sigma_{1}\bbS^{1}\left(r\right)\right)/\left\langle I\right\rangle $,
$r\in\left[\frac{1}{2},1\right]$, where $I$ interchanges the two
suspension points and is an involution on $\bbS^{1}\left(r\right)$.
Such an involution is either the antipodal map or a reflection. When
$I$ is a reflection it fixes two points at distance $\pi r$ apart
which implies that $\diam F=\pi r$ and consequently $r=\frac{1}{2}$.
In $F$ consider two points $e,f$ that form an angle $\theta$ at
$A$. By Toponogov comparison
\[
\cos\abs{ef}\geq\cos\abs{Ae}\cos\abs{Af}+\sin\abs{Ae}\sin\abs{Af}\cos\theta.
\]
So when $\theta\leq\pitwo$ and $\abs{Ae},\abs{Af}\in\left(0,\pitwo\right)$
we have $\abs{ef}<\pitwo$. This shows that $\rad F<\pitwo$, when
$I$ is a reflection and $r=\frac{1}{2}$. When $I$ is an antipodal
map we must have $\theta=\pi r$ to obtain maximal distance between
$e$ and $f$. So when $\abs{Ae},\abs{Af}\in\left(0,\pitwo\right)$
and $r<1$ we have $\abs{ef}<\abs{Ae}+\abs{Af}$ as well as $\abs{ef}\leq\pi-\abs{Ae}-\abs{Af}$
when we measure the distance through $B$. But this also forces $\abs{ef}<\pitwo$.
Thus $F$ is isometric to $\bbR\bbP^{2}\left(1\right)$. This also
forces $X_{2}=\bbS^{2}\left(1\right)$ with $G$ acting as the antipodal
map.

We can now use remark \ref{rem:join_boundary} to conclude that either
$X=\left(X_{1}/\left\langle R\right\rangle \right)*X_{2}$ or $X=\left(\left(X_{1}/\left\langle R\right\rangle \right)*X_{2}\right)/\bbZ_{2}$.
\end{proof}
\begin{rem}
If we combine the constructions in the proof and remark \ref{rem:diam=00003Dpi/2},
then it is possible to obtain a classification for Alexandrov spaces
with $\curv\geq1$, $\rad\geq\pitwo$, and dimension $1,2,3$. It
would be interesting to investigate what results one can obtain for
such spaces in higher dimension when they are not finite quotients
of spheres. Do they, e.g., admit submetries onto $\left[0,\pitwo\right]$.
\end{rem}

\section{Quantified Convexity}

The main theorem about the rigidity of lenses and hemispheres has
counter parts for all convex balls in space forms. To understand this
better we introduce some specially tailored modified distance functions
for the distance to the boundary. As mentioned in the introduction
the results here have appeared in \cite{GL} for lower curvature bounds
$\geq0$, but as we use slightly different modified distance functions
we thought it useful to include our proofs.

Consider the metric ball $\bar{B}_{k}\left(r_{0}\right)$ of radius
$r_{0}$ in the space form of curvature $k$, when $k>0$ we assume
that $r_{0}<\frac{\pi}{2\sqrt{k}}$ so as to only consider strictly
convex balls. The boundary of this ball is totally umbilic and when
using the outward pointing normal the eigenvalue of the shape operator
is denoted $\lambda_{0}>0$. However, it is more natural to consider
the inward pointing normal as that is also the gradient for the distance
to $\partial\bar{B}_{k}\left(r_{0}\right)$ inside $\bar{B}_{k}\left(r_{0}\right)$.
This means that the eigenvalue becomes $-\lambda_{0}$. The specific
formula for $\lambda_{0}$ in terms of $r_{0}$ and $k$ will be given
below. Evidently $\lambda_{0}$ is a measure of the convexity of the
boundary.

Let $r\left(x\right)=\abs{x\partial\bar{B}_{k}\left(r_{0}\right)}$
be the distance to the boundary inside $\bar{B}_{k}\left(r_{0}\right)$
and consider the concentric ball of radius $r_{0}-r$ that consists
of points at distance $\geq r$ from the boundary. Denote by $\lambda\left(r\right)<0$
the eigenvalue of the shape operator of this smaller ball for the
inward pointing normal. It is not hard to see that $\lambda$ satisfies
the Riccati equation:

\[
\dot{\lambda}+\lambda^{2}=-k,\,\lambda\left(0\right)=-\lambda_{0}.
\]
The related function $\phi$ coming from: 
\[
\ddot{\phi}+k\phi=-\lambda_{0},\,\phi\left(0\right)=0,\dot{\phi}\left(0\right)=1,
\]
also satisfies $\lambda\dot{\phi}=\ddot{\phi}$ and can be used to
create a suitable modified distance function $f=\phi\circ r$ satisfying
\[
\mathrm{Hess}f+kf=-\lambda_{0}.
\]

Below we give explicit formulas for $\lambda$ and $\phi$ when $k=0,\pm1$.
With that information it is fairly easy to prove that $\lambda$ and
$f$ satisfy the above equations: 
\begin{example}
When $k=0$ it follows that $r_{0}\lambda_{0}=1$ and 
\[
\lambda\left(r\right)=\frac{1}{r-r_{0}}=\frac{\lambda_{0}}{\lambda_{0}r-1},
\]
\[
\phi\left(r\right)=r-\frac{\lambda_{0}}{2}r^{2}.
\]
\end{example}

\begin{example}
When $k=1$ it follows that $\lambda_{0}=\cot\left(r_{0}\right)$
and 
\[
\lambda=\cot\left(r-r_{0}\right)=\frac{\lambda_{0}\cot r+1}{\lambda_{0}-\cot r},
\]
\[
\phi=\sin r+\lambda_{0}\cos r-\lambda_{0}.
\]
\end{example}

\begin{example}
When $k=-1$ and there are no restrictions on $\lambda_{0}$ we have
\[
\lambda=\begin{cases}
\tanh\left(r-r_{0}\right), & 1>\lambda_{0}=\tanh r_{0},\\
-1, & 1=\lambda_{0},\\
\coth\left(r-r_{0}\right), & 1<\lambda_{0}=\coth r_{0}.
\end{cases}
\]
Here only the last case leads to a ball of finite radius. In all three
cases 
\[
\phi=\sinh r-\lambda_{0}\cosh r+\lambda_{0}.
\]
\end{example}

We say that the boundary in a Riemannian manifold $M$ is $\lambda_{0}$-convex,
when all eigenvalues for the shape operator are $\leq-\lambda_{0}$
with respect to the inward normal. When $\lambda_{0}^{2}>-k$ there
exists a metric ball $\bar{B}_{k}\left(r_{0}\right)$ as above with
$\lambda_{0}$-convex boundary. This will be our comparion model space
when $M$ has sectional curvature $\geq k$.

If $r\left(x\right)=\abs{x\partial M}$ and the boundary is $\lambda_{0}$-convex,
then standard Riccati comparison shows that the super level sets $\left\{ x\in M\mid r\left(x\right)\geq r\right\} $
have $-\lambda\left(r\right)$-convex boundary at points where $r$
is smooth, here $\lambda\left(r\right)$ is given by 
\[
\dot{\lambda}+\lambda^{2}=-k,\,\lambda\left(0\right)=-\lambda_{0}.
\]
This in turn imples that the modified distance function $f=\phi\circ r$
satisfies 
\[
\mathrm{Hess}f+kf\leq-\lambda_{0}
\]
in the support sense everywhere.

In the case of Alexandrov spaces there are several interesting model
spaces aside from the constant curvature balls $\bar{B}_{k}\left(r_{0}\right)$. 
\begin{example}
\label{exa:Alex-model}Let $Y$ be an Alexandrov space with curvature
$\geq1$ and consider the cone of radius $r_{0}$: 
\[
C=C_{k}\left(Y\right)\left(r_{0}\right)=\left\{ \left(t,y\right)\mid y\in Y,\,t\in\left[0,r_{0}\right],\,\left(0,y_{0}\right)\sim\left(0,y_{1}\right)\right\} 
\]
with constant radial curvature $k$. The distance between points $\left(t_{0},y_{0}\right)$
and $\left(t_{1},y_{1}\right)$ is determined by the law of cosines
in constant curvature $k$ with the understanding that $\left(t_{i},y_{i}\right)$
is distance $t_{i}$ from the cone point $\left(0,y\right)$ and that
the angle between $\left(t_{0},y_{0}\right)$ and $\left(t_{1},y_{1}\right)$
at the cone point is given by $\abs{y_{0}y_{1}}_{Y}$. When $k>0$
we also assume that $r_{0}\leq\frac{\pi}{2\sqrt{k}}$ in order for
this to become an Alexandrov space with curvature $\geq k$.

These cones have the same convexity properties as our model space
$\bar{B}_{k}\left(r_{0}\right)$. Specifically, if $f\left(r\left(x\right)\right)=\phi\left(\abs{x\partial C}\right)$
and $c\left(t\right)$ is any quasi-geodesic, then $f\left(t\right)=f\circ c\left(t\right)$
satisfies 
\[
\ddot{f}+kf=-\lambda_{0}.
\]
Further observe that since the distance to the cone point is $d\left(t,y\right)=t$,
we have $d+r=r_{0}$ everywhere on $C$ and the standard modified
distance $h\left(x\right)=\psi\left(d\left(x\right)\right)$, where
$\ddot{\psi}+k\psi=1,\,\psi\left(0\right)=0,\,\dot{\psi}\left(0\right)=0$
satisfies 
\[
\ddot{h}+kh=1
\]
along all quasi-geodesics. 
\end{example}

It follows from \cite[Theorem 1.8 (a)]{A-B} that the metric definition
of convexity given in the introduction implies the following theorem. 
\begin{thm}
Let $X$ be an Alexandrov space with curvature $\geq k$ and boundary
that is $\lambda_{0}$-convex. The radius of $X$ is $\leq r_{0}$
and the modified distance function $f=\phi\circ r\circ\gamma$ satisfies
$\ddot{f}+kf\leq-\lambda_{0}$ in the support sense along all (quasi)-geodesics
$\gamma$ in $X$. 
\end{thm}

This shows in particular that 
\begin{cor}
\label{prop:Comp}For any (quasi)-geodesic $c\left(t\right)$ in $X$
with curvature $\geq k$ and $\lambda_{0}$-convex boundary we have
that $f\left(t\right)=\phi\circ r\circ c\left(t\right)$ satisfies
\[
f\left(t\right)\leq\bar{f}\left(t\right)
\]
for all $t\geq0$ where $\bar{f}\geq0$ and $\bar{f}$ is the unique
solution to: 
\[
\ddot{\bar{f}}+k\bar{f}=-\lambda_{0},\,\bar{f}\left(0\right)=f\left(0\right),\dot{\bar{f}}\left(0\right)=\dot{f}^{+}\left(0\right).
\]
\end{cor}

The second theorem in the introduction can now be proven as follows. 
\begin{thm}
\label{thm:rad-rigid-generic}Let $X$ be an Alexandrov space with
curvature $\geq k$ and boundary that is $\lambda_{0}$-convex, where
$\lambda_{0}^{2}>\max\left\{ -k,0\right\} $. The radius satisfies
$\rad X\leq r_{0}$ with equality only when $X$ is isometric to the
cone $C_{k}\left(S_{s}X\right)\left(r_{0}\right)$ for a suitable
$s\in X$. 
\end{thm}

\begin{proof}
It follows from \cite[Cor 1.9 (1)]{A-B} that $r\left(x\right)\leq r_{0}$
for all $x$, i.e., the inradius is $\leq r_{0}$. This in turn implies
that $f$ is strictly concave (see also \cite[Thm 1.8 (a)]{A-B}).
For our modified distance functions this is obvious when $k\geq0$
as we have $\ddot{f}\leq-kf-\lambda_{0}$. When $k=-1$ this also
shows: 
\[
\ddot{f}\leq f-\lambda_{0}=\sinh r-\coth r_{0}\cosh r=-\frac{\cosh\left(r-r_{0}\right)}{\sinh r_{0}}.
\]
This shows that $X$ has a unique point soul $s\in X$ at maximal
distance $r_{1}\leq r_{0}$ from the boundary. Consider a quasi-geodesic
$c:\left[0,b\right]\rightarrow X$ with $c\left(0\right)=s$ and the
function $f\left(t\right)=f\circ c$. This function satisfies 
\[
\ddot{f}+kf\leq-\lambda_{0},\,f\left(0\right)=\phi\left(r_{1}\right),\dot{f}\left(0\right)\leq0.
\]
By proposition \ref{prop:Comp} $f\leq\bar{f}$, where 
\[
\ddot{\bar{f}}+k\bar{f}=-\lambda_{0},\,\bar{f}\left(0\right)=\phi\left(r_{1}\right),\dot{\bar{f}}\left(0\right)=0.
\]
The explicit form of $\bar{f}$ is as follows 
\[
\bar{f}=\begin{cases}
-\cot r_{0}+\left(\sin r_{1}+\cot r_{0}\cos r_{1}\right)\cos t, & k=1,\\
r_{1}-\frac{r_{1}^{2}}{2r_{0}}-\frac{t^{2}}{2r_{0}}, & k=0,\\
\coth r_{0}-\left(\sinh r_{1}-\coth r_{0}\cosh r_{1}\right)\cosh t, & k=-1.
\end{cases}
\]
These expressions imply that $\bar{f}\left(r_{0}\right)\leq0$, and
that $\bar{f}\left(r_{0}\right)=0$ only occurs when $r_{1}=r_{0}$.
This shows that $b\leq r_{0}$ and consequently that the radius is
$\leq r_{0}$. Moreover, equality can only happen when $r_{1}=r_{0}$,
i.e., the boundary is at constant distance from the soul.

The remainder of the proof is along the same lines as the rigidity
statement in proposition \ref{prop:dimS=00003D0}. If we assume that
$r_{1}=r_{0}$, then this shows more generally that any quasi-geodesic
emanating from the soul must hit the boundary precisely when $t=b=r_{0}$.
Since this agrees with the distance to any point on the boundary any
such quasi-geodesic is a minimal geodesic. In particular, any point
in $X$ lies on a minimal geodesic from the soul to the boundary and
$d\left(x\right)+r\left(x\right)=\abs{xs}+\abs{x\partial X}=r_{0}$
for all $x\in X$. This shows that $\dot{d}=-\dot{r}$ and $\ddot{r}=-\ddot{d}$
along any geodesic in $X$. As both $d$ and $r$ have natural upper
bounds on these second derivatives they also have the same natural
lower bounds. For example when $k=0$ we have 
\begin{eqnarray*}
\ddot{d} & \leq & \frac{1-\dot{d}^{2}}{d}\textrm{ and}\\
\ddot{r} & \leq & \frac{1-\dot{r}^{2}}{r-r_{0}}=-\frac{1-\dot{d}^{2}}{d}.
\end{eqnarray*}
In particular, $\ddot{r}=\lambda\left(r\right)$ and $\ddot{f}+kf=-\lambda_{0}$.
Moreover, with notation as in example \ref{exa:Alex-model} the modified
distance function $h=\psi\circ d$, satisfies $\ddot{h}+kh=1$. This
shows that $X$ is isometric to the cone $C=C_{k}\left(S_{s}X\right)\left(r_{0}\right)$
via the gradient exponential map $\mathrm{gexp}_{s}\left(k;\cdot\right):C\rightarrow X$.
Specifically, in both cases the gradient flow for $h$ from the center
is a flow along minimal geodesics and the distance between points
on these radial geodesics is governed by the equation $\ddot{h}+kh=1$
(see also \cite[Section 2, esp. Thm 2.3.1]{Pe}). 
\end{proof}
\begin{rem}
The rigidity aspect of this proof can also be found in \cite[Cor 1.10 (1)]{A-B}
where the authors establish rigidity for spaces $X$ with curvature
$\geq k$, $\lambda_{0}$-convex boundary, and maximal inradius, i.e.,
the maximal distance to the boundary is $r_{0}$. Note that while
the radius of the lens $L_{\alpha}^{n}$ is $\frac{\pi}{2}$, its
inradius is $\frac{\alpha}{2}$. 
\end{rem}

\begin{cor}
Let $X$ be an Alexandrov space with curvature $\geq k$ and boundary
that is $\lambda_{0}$-convex, where $\lambda_{0}^{2}>\max\left\{ -k,0\right\} $.
If $\vol_{n-1}\partial X=\vol_{n-1}\partial\bar{B}_{k}\left(r_{0}\right)$,
then $X$ is isometric to $\bar{B}_{k}\left(r_{0}\right)$. 
\end{cor}

\begin{proof}
First observe that for any Alexandrov space with curvature $\geq k$
and $\rad\leq r$ we can use Petrunin's solution of Lytchak's problem
(see \cite[3.3.5]{Pe}) to show that $\vol_{n-1}\partial X\leq\vol\partial C_{k}\left(S_{p}X\right)\left(r\right)$.
In particular, if $\rad X<r_{0}$, then it follows that $\vol_{n-1}\partial X<\vol_{n-1}\partial\bar{B}_{k}\left(r_{0}\right)$.
In particular, the condition $\vol_{n-1}\partial X=\vol_{n-1}\partial\bar{B}_{k}\left(r_{0}\right)$
implies that $\rad X=r_{0}$ and the above theorem that $X=C_{k}\left(S_{s}X\right)\left(r_{0}\right)$.
Finally, $\vol_{n-1}\left(\partial C_{k}\left(S_{s}X\right)\left(r_{0}\right)\right)=\vol_{n-1}\partial\bar{B}_{k}\left(r_{0}\right)$
only when $S_{s}X=S^{n-1}\left(1\right)$. This shows that 
\[
X=C_{k}\left(S_{s}X\right)\left(r_{0}\right)=C_{k}\left(S^{n-1}\left(1\right)\right)\left(r_{0}\right)=\bar{B}_{k}\left(r_{0}\right).
\]
\end{proof}

\end{document}